\documentclass[11pt]{article}

\usepackage[T1]{fontenc}
\usepackage[utf8]{inputenc}
\usepackage{amsmath,amssymb,amsthm,mathtools}
\usepackage{microtype}
\usepackage[hidelinks]{hyperref}

\newtheorem{theorem}{Theorem}[section]
\newtheorem{lemma}[theorem]{Lemma}
\newtheorem{proposition}[theorem]{Proposition}
\newcommand{\T}{\mathbb T}
\newcommand{\R}{\mathbb R}
\newcommand{\Z}{\mathbb Z}
\newcommand{\Q}{\mathbb Q}
\newcommand{\usual}{\tau_{\mathrm{usual}}}
\newcommand{\torusnorm}[1]{\lVert #1\rVert_{\T}}
\DeclareMathOperator{\dist}{dist}

\title{Topological Vector Group Topologies Between the Minimal Topology and the Usual Topology on the Real Line}
\author{Irina Yaroshevskaya}
\date{}

\begin{document}
\maketitle

\begin{abstract}
For every positive sequence that tends to zero faster than every fixed exponential, we construct a Hausdorff topological Vector Group topology on the additive group of real numbers. It lies strictly between the minimal Hausdorff topological Vector Group topology and the usual topology. The construction is illustrated by factorial powers, quadratic exponential decay, and prime radicals divided by a quadratic exponential. We also record a finite scalar covering criterion for comparing two such topologies.
\end{abstract}



\section{Introduction}

A \emph{topological Vector Group} (TVG) topology in this article is a group topology on a real vector space for which multiplication by each fixed real scalar is continuous. Joint continuity of scalar multiplication is not required. The topology $t$ was introduced by I.~B.~Yaroshevskaya in 1975~\cite{Y1975}; it was proved minimal among Hausdorff TVG topologies on $\R$ in 1981~\cite{Y1981}. We denote the usual topology on $\R$ by $\usual$.

Let $\T=\R/\Z$ be the circle group and write
\[
  \torusnorm{x}=\dist(x,\Z)=\inf_{m\in\Z}|x-m|.
\]
For a finite nonempty set $F\subset\R\setminus\{0\}$ and $\varepsilon>0$, put
\[
  T(F,\varepsilon)=
  \{x\in\R:\torusnorm{\lambda x}<\varepsilon
  \text{ for every }\lambda\in F\}.
\]
The sets $T(F,\varepsilon)$ form a neighborhood base at zero for $t$.

Let $(r_n)_{n\geq1}$ be a positive sequence such that
\begin{equation}\label{eq:decay}
  \frac{-\log r_n}{n}\longrightarrow\infty.
\end{equation}
Equivalently, for every $c>0$ one has $r_n\leq e^{-cn}$ for all sufficiently large $n$. For a finite nonempty set $A\subset\R\setminus\{0\}$ define
\[
  p_A(x)=\sup_{\substack{a\in A\\ n\geq1}}
  \torusnorm{a r_nx},
  \qquad
  U(A,\varepsilon)=\{x\in\R:p_A(x)<\varepsilon\}.
\]
Let $\sigma_{(r_n)}$ be the translation-invariant topology generated by these sets.

\begin{theorem}[General construction]\label{thm:main}
The topology $\sigma_{(r_n)}$ is a Hausdorff TVG topology and
\[
  t\subsetneq\sigma_{(r_n)}\subsetneq\usual.
\]
\end{theorem}

\section{Proof of the general construction}

\subsection{The topology and the TVG property}

The circle distance satisfies
$\torusnorm{u+v}\leq\torusnorm{u}+\torusnorm{v}$ and
$\torusnorm{-u}=\torusnorm{u}$. Taking suprema gives
\[
  p_A(x+y)\leq p_A(x)+p_A(y),
  \qquad p_A(-x)=p_A(x).
\]
Consequently,
\[
 U(A,\varepsilon/2)+U(A,\varepsilon/2)
 \subset U(A,\varepsilon),
\]
each $U(A,\varepsilon)$ is symmetric, and
\[
 U(A\cup B,\min\{\varepsilon,\delta\})
 \subset U(A,\varepsilon)\cap U(B,\delta).
\]
If $y\in U(A,\varepsilon)$, set
$\eta=\varepsilon-p_A(y)>0$. For $h\in U(A,\eta)$,
\[
 p_A(y+h)\leq p_A(y)+p_A(h)<\varepsilon.
\]
Hence $y+U(A,\eta)\subset U(A,\varepsilon)$. Together with the preceding relations, this proves that the translates of the sets $U(A,\varepsilon)$ form a group-topology base. If $c\neq0$, then
\[
  p_A(cx)=p_{cA}(x),
  \qquad cA=\{ca:a\in A\}.
\]
Thus multiplication by $c$ and its inverse, multiplication by $1/c$, are continuous. Multiplication by zero is constant and continuous. Hence $\sigma_{(r_n)}$ is a TVG topology.

If $x\neq0$, choose $a=(2r_1x)^{-1}$. Then
$\torusnorm{ar_1x}=1/2$, so $x\notin U(\{a\},1/4)$. Therefore the intersection of all zero-neighborhoods is $\{0\}$, and the topology is Hausdorff.

\subsection{Comparison with the usual topology}

Condition~\eqref{eq:decay} implies $r_n\to0$. Hence
\[
 r^*=\sup_{n\geq1}r_n<\infty.
\]
For finite $A$, let $M_A=\max_{a\in A}|a|$. Then
\[
 p_A(x)\leq M_A r^*|x|.
\]
Let $y\in U(A,\varepsilon)$ and put
$\eta=\varepsilon-p_A(y)>0$. If
$|h|<\eta/(M_Ar^*)$, then $p_A(h)<\eta$, and therefore
$p_A(y+h)<\varepsilon$. Thus every point of $U(A,\varepsilon)$ has an ordinary open interval contained in $U(A,\varepsilon)$. Every $\sigma_{(r_n)}$-open set is therefore usual-open, and
$\sigma_{(r_n)}\subseteq\usual$.

\subsection{Unboundedness of basic neighborhoods}

\begin{lemma}[Finite approximation]\label{lem:dirichlet}
For real numbers $\beta_1,\ldots,\beta_d$ and an integer $H\geq2$, there is an integer $q$ such that
\[
 1\leq q\leq H^d,
 \qquad
 \torusnorm{q\beta_\ell}<\frac1H
 \quad(1\leq\ell\leq d).
\]
\end{lemma}

\begin{proof}
Place the $H^d+1$ points
\[
 (\{m\beta_1\},\ldots,\{m\beta_d\}),
 \qquad 0\leq m\leq H^d,
\]
into the $H^d$ half-open boxes of side $1/H$ partitioning $[0,1)^d$. Two points, with indices $m_0<m_1$, lie in the same box. Therefore, for every coordinate $\ell$,
\[
 \big|\{m_1\beta_\ell\}-\{m_0\beta_\ell\}\big|<\frac1H.
\]
With $q=m_1-m_0$, this gives $1\leq q\leq H^d$ and the asserted circle-distance inequalities.
\end{proof}

\begin{lemma}\label{lem:unbounded}
Every basic neighborhood $U(A,\varepsilon)$ is unbounded.
\end{lemma}

\begin{proof}
Write $A=\{a_1,\ldots,a_k\}$ and prescribe $T_0>0$. Put
\[
 r^*=\sup_{n\geq1}r_n,
 \qquad M=\max_{1\leq i\leq k}|a_i|,
 \qquad b=Mr^*,
 \qquad s_N=\sup_{j>N}r_j.
\]
Choose positive integers $L,H$ such that $L>bT_0$ and $L/H<\varepsilon$. For fixed $H$ and $k$, condition~\eqref{eq:decay} gives
\begin{equation}\label{eq:tail}
 H^{kN}s_N\longrightarrow0.
\end{equation}
Indeed, choose $c>k\log H$. For all sufficiently large $j$, $r_j\leq e^{-cj}$, and hence $s_N\leq e^{-c(N+1)}$ for all sufficiently large $N$.

Choose $N$ so large that
\[
 \frac{L}{r^*}H^{kN}s_N<\varepsilon.
\]
Apply Lemma~\ref{lem:dirichlet} to the $kN$ numbers
\[
 \frac{a_ir_j}{b},
 \qquad 1\leq i\leq k,
 \quad 1\leq j\leq N,
\]
and define $x=Lq/b$. Since $q\geq1$ and $L>bT_0$, we have $x>T_0$. If $j\leq N$, integer multiplication in $\T$ gives
\[
 \torusnorm{a_ir_jx}
 \leq L\torusnorm{q\frac{a_ir_j}{b}}
 <\frac{L}{H}<\varepsilon.
\]
If $j>N$, then
\[
 \torusnorm{a_ir_jx}
 \leq |a_ir_jx|
 \leq \frac{L}{r^*}H^{kN}s_N
 <\varepsilon.
\]
Thus $x\in U(A,\varepsilon)$ and $x>T_0$. Since $T_0$ was arbitrary, the neighborhood is unbounded.
\end{proof}

No bounded ordinary interval can contain a basic $\sigma_{(r_n)}$-neighborhood. Therefore
$\sigma_{(r_n)}\subsetneq\usual$.

\subsection{Comparison with the minimal topology}

Fix $r_1>0$. For a finite $F\subset\R\setminus\{0\}$, put
$A=\{\lambda/r_1:\lambda\in F\}$. The constraint with $n=1$ gives
\[
 U(A,\varepsilon)\subset T(F,\varepsilon).
\]
Hence $t\subseteq\sigma_{(r_n)}$.

\begin{lemma}[Finite-translate property of $t$]\label{lem:syndetic}
For every finite nonempty $F\subset\R\setminus\{0\}$ and every $\varepsilon>0$, there is a finite set $C\subset\R$ such that
\[
 \R=C+T(F,\varepsilon).
\]
\end{lemma}

\begin{proof}
Define
\[
 \phi_F:\R\longrightarrow\T^F,
 \qquad
 \phi_F(x)=(\lambda x+\Z)_{\lambda\in F},
\]
and let $G$ be the closure of $\phi_F(\R)$. The group $G$ is compact. For $\delta>0$, let
\[
 B_\delta=
 \{z\in G:\torusnorm{z_\lambda}<\delta
 \text{ for every }\lambda\in F\}.
\]
Compactness gives $h_1,\ldots,h_s\in G$ such that
$G\subset\bigcup_{j=1}^s(h_j+B_{\varepsilon/2})$.
Because $\phi_F(\R)$ is dense in $G$, choose $c_j\in\R$ with
$h_j\in\phi_F(c_j)+B_{\varepsilon/2}$. The triangle inequality yields
\[
 h_j+B_{\varepsilon/2}
 \subset\phi_F(c_j)+B_\varepsilon.
\]
Pulling the resulting finite cover back under $\phi_F$ gives
$\R=\{c_1,\ldots,c_s\}+T(F,\varepsilon)$.
\end{proof}

Consider $W=U(\{1\},1/4)$. We show that no finite family of translates of $W$ covers $\R$. Let $C\subset\R$ be finite and nonempty. Choose $n$ so large that
\[
 r_n\max_{c\in C}|c|<\frac18,
\]
and set $x=(2r_n)^{-1}$. For every $c\in C$,
\[
 \torusnorm{r_n(x-c)}
 =\torusnorm{\frac12-r_nc}>\frac14.
\]
Therefore $x-c\notin W$ for every $c\in C$, so $x\notin C+W$. The assertion is also immediate when $C$ is empty.

If $W$ were a $t$-neighborhood, it would contain some $T(F,\varepsilon)$; Lemma~\ref{lem:syndetic} would then imply that finitely many translates of $W$ cover $\R$, a contradiction. Hence
$t\subsetneq\sigma_{(r_n)}$. Together with the previous subsections, this completes the proof of Theorem~\ref{thm:main}.

\section{Three examples}

\subsection{Factorial powers}

Let $0<q<1$ and set $r_n=q^{n!}$. Then
\[
 \frac{-\log r_n}{n}
 =(n-1)!\log\frac1q\longrightarrow\infty.
\]
Theorem~\ref{thm:main} gives a topology $\sigma_q$ with
$t\subsetneq\sigma_q\subsetneq\usual$. The parameter $q$ may be transcendental, algebraic, or rational. For example, $q=1/2$ works although the resulting sequence is rational and is not linearly independent over $\Q$.

\subsection{Quadratic exponential decay}

Set $r_n=e^{-n^2}$. Since
\[
 \frac{-\log r_n}{n}=n\longrightarrow\infty,
\]
the resulting topology again lies strictly between $t$ and $\usual$. More generally, the same proof applies to
$r_n=e^{-cn^\alpha}$ whenever $c>0$ and $\alpha>1$.

\subsection{Prime radicals divided by a quadratic exponential}

Let $p_n$ be the $n$th prime and define
\[
 r_n=\frac{\sqrt{p_n}}{2^{n^2}}.
\]
Bertrand's postulate implies $p_n\leq2^n$, and therefore
\[
 \frac{-\log r_n}{n}
 =n\log2-\frac{\log p_n}{2n}
 \geq\left(n-\frac12\right)\log2
 \longrightarrow\infty.
\]
Thus Theorem~\ref{thm:main} applies. The numbers $\sqrt{p_n}$ are linearly independent over $\Q$: for any finite subfamily, the sign-changing automorphisms of the corresponding multiquadratic extension isolate each coefficient in a rational linear relation. This algebraic fact is not needed for the topological construction.

\section{Dependence on the sequence}

The topology depends only on the compact scalar set
\[
 K_r=\{0\}\cup\{r_n:n\geq1\},
\]
not on the order or repetition of its nonzero elements. Write $\sigma_K$ for the topology obtained from a compact scalar set $K$ by using the seminorms
\[
 p_{A,K}(x)=\sup_{\substack{a\in A\\u\in K}}
 \torusnorm{a\cdot u\cdot x}.
\]

\begin{proposition}[Finite scalar covering criterion]\label{prop:cover}
If
\[
 K_1\subset c_1K_2\cup\cdots\cup c_sK_2
\]
for finitely many nonzero real scalars $c_1,\ldots,c_s$, then
$\sigma_{K_1}\subseteq\sigma_{K_2}$.
\end{proposition}

\begin{proof}
Given a finite scalar set $A$, put
$A'=\bigcup_{j=1}^s c_jA$. Every constraint appearing in
$p_{A,K_1}$ then also appears among the constraints in $p_{A',K_2}$.
Consequently,
\[
 U_{K_2}(A',\varepsilon)
 \subset U_{K_1}(A,\varepsilon),
\]
which is precisely the asserted inclusion of topologies.
\end{proof}

For $0<q<1$, set
\[
 K_q=\{0\}\cup\{q^{n!}:n\geq1\}.
\]

\begin{proposition}\label{prop:factorial-rigidity}
If
\[
 K_{q_1}\subset c_1K_{q_2}\cup\cdots\cup c_sK_{q_2}
\]
for finitely many nonzero real scalars, then $q_1=q_2$.
\end{proposition}

\begin{proof}
Put $\alpha=\log(1/q_1)>0$ and $\beta=\log(1/q_2)>0$. By the pigeonhole principle, for one positive scalar $c$ there are infinitely many pairs $(n,m)$ such that
\[
 q_1^{n!}=c q_2^{m!}.
\]
Thus
\[
 \alpha n!=\beta m!-\log c,
 \qquad
 \frac{m!}{n!}\longrightarrow\frac{\alpha}{\beta}
\]
along these pairs. If $m\geq n+1$, then $m!/n!\geq n+1$; if $m\leq n-1$, then $m!/n!\leq1/n$. Since $\alpha/\beta$ is finite and positive, necessarily $m=n$ for all sufficiently large pairs. Taking the limit gives $\alpha=\beta$, and hence $q_1=q_2$.
\end{proof}

Proposition~\ref{prop:factorial-rigidity} distinguishes the scalar sets under the finite-cover relation of Proposition~\ref{prop:cover}. Whether equality $\sigma_{q_1}=\sigma_{q_2}$ alone forces $q_1=q_2$ is not resolved here.

\section{Historical context and related literature}

Yaroshevskaya introduced $t$ in 1975 and proved that it makes $\R$ a Hausdorff TVG~\cite{Y1975,Y1975English}. Her paper received in 1979 and published in 1981 proved minimality among Hausdorff TVG topologies on $\R$~\cite{Y1981}, using the circle group $\T=\R/\Z$ and almost-periodic functions.

Raikov~\cite{Raikov1968} and Kenderov~\cite{Kenderov1970} developed the broader theory of topological vector groups. Dergachev~\cite{Dergachev1978} later described $t$ through usual-continuous almost-periodic functions and used it to construct a connected TVG that is not a topological vector space. Hejcman studied TVG topologies on the real line~\cite{Hejcman1988}. Dom\'inguez and Tarieladze~\cite{DominguezTarieladze2008} describe the same minimal topology through the usual-continuous characters of $\R$ and discuss its precompactness and scalar-continuity properties. The Kronecker--Weyl and character-theoretic references~\cite{Bailleul2022,DGT2012} provide alternative tools used in longer treatments of this construction.

\section*{Acknowledgement}

The author acknowledges the assistance of OpenAI's Codex and Anthropic's Claude in mathematical discussion, review of proofs, improvement of exposition, and document preparation. The author assumes full responsibility for the final content and any remaining errors.

\end{document}